\documentclass[12pt,a4paper,oneside,reqno]{amsproc}
\usepackage{amsmath} 
\usepackage{amssymb,amsfonts,mathrsfs}
\usepackage[colorlinks=true,linkcolor=blue,citecolor=blue]{hyperref}
\usepackage[driver=pdftex,lmargin=2.5cm,rmargin=2.5cm,heightrounded=true,centering]{geometry}
\usepackage{math50}
\usepackage{local}
\usepackage{epsfig}  
\usepackage{graphicx}  
\usepackage{xy}
\xyoption{all}

\begin{document}
\title[Fractional Schr\"{o}dinger Equations]
{Multiple Semiclassical Standing Waves for Fractional Nonlinear Schr\"{o}dinger Equations}
\author{Guoyuan Chen}
\address{School of Mathematics and Statistics, Zhejiang University of Finance \& Economics, Hangzhou 310018, Zhejiang, P. R. China}
\email{gychen@zufe.edu.cn}

\newcommand{\optional}[1]{\relax}
\setcounter{secnumdepth}{3}
\setcounter{section}{0} \setcounter{equation}{0}
\numberwithin{equation}{section}

\date{\today}
\begin{abstract}
Via a Lyapunov-Schmidt reduction, we obtain multiple semiclassical solutions to a class of fractional nonlinear Schr\"odinger equations. Precisely, we consider
\begin{equation*}
\varepsilon^{2s}(-\Delta)^{s}u+u+V(x)u=|u|^{p-1}u,\quad  u\in H^s(\mathbf R^n),
\end{equation*}
where $0<s<1$, $n>4-4s$, $1<p<\frac{n+2s}{n-2s}$ (if $n>2s$) and $1<p<\infty$ (if $n\le 2s$), $V(x)$ is a non-negative potential function. If $V$ is a sufficiently smooth bounded function with a non-degenerate compact critical manifold $M$, then, when $\varepsilon$ is sufficiently small, there exist at least $l(M)$ semiclassical solutions, where $l(M)$ is the cup length of $M$.
\end{abstract}
\maketitle

\section{Introduction}

Fractional Schr\"odinger equations are derived from the path integral over L\'{e}vy trajectories. It can be applied, for example, to describe the orbits radius for hydrogen-like atoms. (For more details of physical background, see, for example, \cite{La:FSE} and the references therein.)

We study the fractional nonlinear Schr\"odinger equation of form
\begin{equation}\label{e:evolutionEqn}
i\varepsilon \frac{\partial\psi}{\partial t}=(-\varepsilon^2\Delta)^s\psi+V(x)\psi-|\psi|^{p-1}\psi \quad\mbox{in }\mathbf R^n,
\end{equation}
where $\varepsilon$ is a small positive constant which is corresponding to the Planck constant, $(-\Delta)^s$, $0<s<1$, is the fractional Laplacian, $V(x)$ is a potential function, and $p>1$.

We shall look for the so-called standing wave solutions which are of form
\begin{equation*}\label{e:standingwave}
\psi(x,t)=e^{(i/\varepsilon)Et}v(x),
\end{equation*}
where $v$ is a real-valued function depending only on $x$ and $E$ is some constant in $\mathbf R$. The function $\psi$ solves (\ref{e:evolutionEqn}) provided the standing wave $v(x)$ satisfies
\begin{equation}\label{e:stationally}
(-\varepsilon^2\Delta)^sv+(V(x)+E)v-|v|^{p-1}v=0\quad\mbox{in } \mathbf R^n.
\end{equation}

In what follows, we assume that $E=1$ and $p$ is subcritical. That is, we will study the following equation:
\begin{equation}\label{e:main-equation}
\varepsilon^{2s}(-\Delta)^{s}u+u+V(x)u=|u|^{p-1}u,\quad  u\in H^s(\mathbf R^n),
\end{equation}
where $0<s<1$, and $1<p<\frac{n+2s}{n-2s}$ for $n>2s$, and, $1<p<\infty$ for $n\le2s$.

In quantum mechanics, when $\varepsilon$ tends to zero, the existence and multiplicity of solutions to (\ref{e:main-equation}) is of importance. We will find multiple solutions $u_{\varepsilon}$ of (\ref{e:main-equation}) that concentrate near some point $x_0\in \mathbf R^n$ as $\varepsilon\to 0$. By this we mean that, for all $x\in \mathbf R^n\setminus \{x_0\}$, $u_{\varepsilon}(x)\to 0$ as $\varepsilon\to 0$. Such kind of solutions are so-called semiclassical standing waves or spike pattern solutions.

When $s=1$, Equation (\ref{e:main-equation}) is a classical nonlinear Schr\"odinger equation and the existence of semiclassical standing wave solutions was established by Floer and Weinstein \cite{FW:NWPCS}, and then Oh \cite{Oh:CMP89, Oh:CMP90}. There is a large mount of research on this subject in the past two decades. We refer for example to the (far from complete) list of papers \cite{ABC:ARMA97, AMS:MRNSE, AMN:SPEE, BL:MPEPS, BL:EPSSE, DF:LMP, BW:SWCF, G:NSPS, Li:SPEE, DW:CCNSE, CL:TMNA97, DF:MA02, G:CPDE96, KW:ADE00, R:ZAMP92, W:CMP93, MMM:SSE} and the references therein.

When $s\in (0,1)$, the existence of semiclassical solution to Equations (\ref{e:main-equation}) was obtained by D\'{a}vila, del Pino and Wei \cite{DDW:CSWFSE}, Chen and Zheng \cite{CZ:CPFSE}. Precisely, by a Lyapunov-Schimdt reduction, \cite{DDW:CSWFSE} proved that if $V$ is a sufficiently smooth positive function with non-degenerate critical points $\xi_1,\xi_2,\cdots,\xi_k$ and satisfies some degree conditions around these points, then there exists a solution of (\ref{e:main-equation}) concentrating to these $k$ critical points. (See \cite{CZ:CPFSE} for the case $k=1$ with more technical conditions.) Further, in \cite{FaMaVa14:GSCP}, Fall, Mahmoudi and Valdinoci proved that if there exist semiclassical solutions to (\ref{e:main-equation}) as $\varepsilon\to 0$, then the concentration points must be critical points of $V$.

Moreover, we should mention that the concentration phenomena for fractional Schr\"odinger equations on bounded domain with Dirichlet condition were investigated by D\'{a}vila, del Pino, Dipierro and Valdinoci \cite{DaDPDiVa14}.

In this paper, we mainly investigate existence and multiplicity of semiclassical standing wave solutions to Equation (\ref{e:main-equation}) when $V$ has non-isolated critical points. More precisely, we have the following theorem.
\begin{theorem}\label{t:main}
Let $0<s<1$, $n> 4-4s$. Suppose that $V$ is a non-negative function in $C^3_b(\mathbf R^n)$ with a non-degenerate smooth compact critical manifold $M$. Then for $\varepsilon>0$ small, Equation (\ref{e:main-equation}) has at least $l(M)$ solutions concentrating near points of $M$.
\end{theorem}
Here $l(M)$ denotes the cup length of $M$ (see Section \ref{ss:abstract} below) and $$C^3_b(\mathbf R^n)=\{v\in C^3(\mathbf R^n) \mid \partial^{J}v \mbox{ is bounded on }\mathbf R^n \mbox{ for all } |J|\le 3\}.$$
The non-degeneracy of a critical manifold is in the sense of Bott \cite{Bo:AM57}. Precisely, we say that a critical manifold $M$ of $V$ is non-degenerate if, for every $x\in M$, the kernel of $D^2f(x)$ equals to $T_x M$.

\begin{remark}
When $s=1$, the result of this theorem was obtained by Ambrosetti, Malchiodi and Secchi \cite{AMS:MRNSE}.
\end{remark}

\begin{remark}
Since the unique positive solution (up to translation) to the standard equation decays as $1/(1+|x|^{n+2s})$ (see for example \cite{FS:URSFL, FQT:PSNFS} or Theorem \ref{t:unda} below), we should technically assume that $0<s<1$ and $n> 4-4s$ to make some necessary integrals convergent (see the proof of Lemma \ref{l:vxz} below). Based on our observation, this assumption is essential since the decay estimate of the unique standard solution is optimal. We should also note that when $s\to 1$, there is no restriction on the dimension $n$. This is the same as the classical case $s=1$.
\end{remark}

\begin{remark}
Note that the assumption $V\ge 0$ on $\mathbf R^n$ is not essential. In fact, a similar argument as in Section \ref{sb:cuf} below implies that the condition  $\inf (1+V)>0$ is sufficient. Without loss of generality, in what follows we assume that $V(0)=0$ for simplicity.
\end{remark}

Our proof relies on a singular perturbation argument as in \cite{AMS:MRNSE}. More precisely, by the change of variable $x\to \varepsilon x$, Equation (\ref{e:main-equation}) becomes
\begin{equation}\label{e:change}
(-\Delta)^{s}u+u+V(\varepsilon x)u=|u|^{p-1}u.
\end{equation}
Solutions of (\ref{e:change}) are the critical points $u\in H^s(\mathbf R^n)$ of the functional
\begin{equation}
f_{\varepsilon}(u)=f_0(u)+\frac{1}{2}\int_{\mathbf R^n}V(\varepsilon x)u^2dx,
\end{equation}
where
\begin{equation}\label{e:f0u}
f_0(u)=\frac{1}{2}\|u\|_s^2-\frac{1}{p+1}\int_{\mathbf R^n}|u|^{p+1}dx.
\end{equation}
Here $\|\cdot\|_s$ denotes the norm in $H^s(\mathbf R^n)$.
We should note that $f_{\varepsilon }\in C^2(H^s(\mathbf R^n))$.
We will find the solutions of (\ref{e:change}) near the solutions of
\begin{equation}\label{e:ss}
(-\Delta)^{s}u+u+V(\varepsilon \xi)u=|u|^{p-1}u,
\end{equation}
for some $\xi\in \mathbf R^n$ to be fixed. The solutions of (\ref{e:ss}) are critical points of the following functional
\begin{equation}\label{f:xi}
F_{\varepsilon, \xi}(u)=f_0(u)+\frac{1}{2}V(\varepsilon \xi)\int_{\mathbf R^n}u^2dx.
\end{equation}
Since (\ref{f:xi}) has a term of $V$, $F_{\varepsilon, \xi}$ inherits the topological features of the critical manifold $M$ of $V$. Therefore, if we consider $f_{\varepsilon}$ as a perturbation of $F_{\varepsilon,\xi}$, multiple solutions to (\ref{e:change}) will be found by a multiplicity theorem from \cite{Ch:IDMT} (see Theorem \ref{t:abstract} below).

Nevertheless, a direct application of the arguments in \cite{AMS:MRNSE} to our problem is impossible. There are two reasons which make our proof much more complicated. Firstly, unlike the Laplacian $-\Delta$, the fractional Laplacian $(-\Delta)^s$, $0<s<1$, is nonlocal. For this reason, when $0<s<1$, the classical local techniques as in $s=1$ case (see \cite{AMS:MRNSE}) can not be used any more. For instance, instead of using the classical method in \cite{AMS:MRNSE} which depends on the locality of $-\Delta$ essentially, we employ a functional analysis approach to prove the invertibility of $D^2f_{\varepsilon}$ (see Section \ref{s:inertibility} below). Secondly, the standard solution $U$ to unperturbed fractional Schr\"odinger equation ($V\equiv 0$ in Equation (\ref{e:main-equation})) decays only as $1/(1+|x|^{n+2s})$ (see Section \ref{ss:se} below), especially it does not decay exponentially as in $s=1$ case. Therefore, to ensure the necessary functions in certain Sobolev spaces on $\mathbf R^n$ and to recover the estimates for Lyapunov-Schmidt reduction, we need more detailed and involved analysis than the classical case (see Section \ref{s:estimates}, \ref{s:inertibility} and \ref{s:reduction} below).

Our paper is organized as follows. In Section \ref{s:preliminaries}, we recall the notations of Fractional Sobolev spaces, some basic properties of standard equation which is obtained by \cite{FL:UNGS, FS:URSFL, FQT:PSNFS}. Moreover, we formulate the functional corresponding Equation (\ref{e:main-equation}), and construct the critical manifold of the functional (\ref{f:xi}). In Section \ref{s:estimates}, some useful estimates are showed for further reference. In Section \ref{s:inertibility}, we prove the invertibility of linearized operator at the points on critical manifold of $F_{\varepsilon,\xi}$. In Section \ref{s:reduction}, we apply the Lyapunov-Schmidt reduction method to our functional. In Section \ref{s:proof}, we complete the proof of Theorem \ref{t:main}.


\section{Preliminaries}\label{s:preliminaries}

In this section, we recall some results on fractional Laplacian, fractional Sobolev spaces and some uniqueness, non-degeneracy and decay results for solutions to the standard Schr\"odinger equations.

\subsection{Fractional Laplacian and fractional order Sobolev spaces}
For further references, we recall some basic facts involving fractional Laplacian and fractional order Sobolev spaces. For more details, see, for example, \cite{Ad:SS}, \cite{Sh:POST}, \cite{NPV:HG}, \cite{Caffarelli&Silvestre07}.

Mathematically, $(-\Delta)^s$ is defined as
$$
(-\Delta)^s u = C(n, s)\mbox{P.V.} \int_{\mathbf{R}^n}\frac{u(x) - u(y)}{|x - y|^{n + 2s}}dy = C(n, s)\lim_{\delta\to 0^+}\int_{\mathbf R^n\setminus B_{\delta}(x)}\frac{u(x) - u(y)}{|x - y|^{n + 2s}}dy.
$$
Here P. V. is a commonly used abbreviation for `in the principal value sense' and $C(n, s) = \pi^{-(2s + n/2)}\frac{\Gamma(n/2 + s)}{\Gamma(-s)}$.
It is well known that $(-\Delta)^s$ on $\mathbf{R}^{n}$ with $s\in (0, 1)$ is a nonlocal operator.

When $s\in (0, 1)$, the space $H^{s}(\mathbf{R}^{n}) = W^{s, 2}(\mathbf{R}^n)$ is defined by
\begin{eqnarray*}
H^{s}(\mathbf{R}^{n})& = &\left\{u\in L^2(\mathbf{R}^2): \frac{|u(x) - u(y)|}{|x - y|^{\frac{n}{2} + s}}\in L^{2}(\mathbf{R}^n\times\mathbf{R}^n)\right\}\\
& = & \left\{u\in L^2(\mathbf{R}^2): \int_{\mathbf{R}^n}(1 + |\zeta|^{2s})|\mathcal{F}u(\zeta)|^2d\zeta < +\infty\right\}
\end{eqnarray*}
and the inner product is
\begin{eqnarray*}
\langle u,v\rangle_{s} & := &\int_{\mathbf{R}^n}uvdx + \int_{\mathbf{R}^n}\int_{\mathbf{R}^n}\frac{(u(x) - u(y))(v(x)-v(y))}{|x - y|^{n + 2s}}dxdy.
\end{eqnarray*}
Let
$$
[u]_{s} := [u]_{H^{s}(\mathbf{R}^{n})} = \left(\int_{\mathbf{R}^n}\int_{\mathbf{R}^n}\frac{|u(x) - u(y)|^2}{|x - y|^{n + 2s}}dxdy\right)^\frac{1}{2}
$$
be the Gagliardo (semi) norm of $u$. The following identity yields the relation between the fractional operator $(-\Delta)^s$ and the fractional Sobolev space $H^{s}(\mathbf{R}^{n})$,
$$
[u]_{H^{s}(\mathbf{R}^{n})} = C\left(\int_{\mathbf{R}^n}|\zeta|^{2s}|\mathcal{F}u(\zeta)|^2d\zeta\right)^{\frac{1}{2}} = C\|(-\Delta)^{\frac{s}{2}}u\|_{L^2(\mathbf{R}^n)}
$$
for a suitable positive constant $C$ depending only on $s$ and $n$.

When $s > 1$ and it is not an integer we write $s = m + \sigma$, where $m$ is an integer and $\sigma\in (0, 1)$. In this case the space $H^{s}(\mathbf{R}^{n})$
consists of those equivalence classes of functions $u\in H^{m}(\mathbf{R}^{n})$ whose distributional derivatives $D^{J} u$, with $|J| = m$, belong to $H^{\sigma}(\mathbf{R}^{n})$, namely
\begin{eqnarray*}
H^{s}(\mathbf{R}^{n}) = \left\{u\in H^{m}(\mathbf{R}^{n}): D^{J} u\in H^{\sigma}(\mathbf{R}^{n}) \mbox{\,\,for any\,\,}J \mbox{\,\,with\,\,} |J| = m\right\}
\end{eqnarray*}
and this is a Banach space with respect to the norm
\begin{eqnarray*}
\|u\|_{s} := \|u\|_{H^{s}(\mathbf{R}^{n})} = \left(\|u\|^2_{H^{m}(\mathbf{R}^{n})} + \displaystyle\sum_{|J| = m}\|D^{J} u\|^2_{H^{\sigma}(\mathbf{R}^{n})}\right)^\frac{1}{2}.
\end{eqnarray*}
Clearly, if $s = m$ is an integer, the space $H^{s}(\mathbf{R}^{n})$ coincides with the usual Sobolev space $H^{m}(\mathbf{R}^{n})$. By this notation, we denote the norm of $L^2(\mathbf R^n)$ by $\|\cdot\|_0$.

For a general domain $\Omega$, the space $H^s(\Omega)$ can be defined similarly.

On the Sobolev inequality and the compactness of embedding, one has
\begin{theorem}\cite{Ad:SS}\label{l:embedding}
Let $\Omega$ be a domain with smooth boundary in $\mathbf{R}^n$. Let $s > 0$, then
\begin{enumerate}
\item[(a)]If $n > 2s$, then $H^{s}(\Omega)\hookrightarrow L^{r}(\Omega)$ for $2\leq r \leq 2n/(n - 2s)$,
\item[(b)]If $n = 2s$, then $H^{s}(\Omega)\hookrightarrow L^{r}(\Omega)$ for $2\leq r < \infty$,
\end{enumerate}
\end{theorem}
\begin{theorem}\cite{Sh:POST}\label{l:compactness}
Let $s>s'$ and $\Omega$ be a bounded domain with smooth boundary in $\mathbf R^n$. Then the embedding operator
\begin{equation*}
i_s^{s'}:H^s(\Omega)\to H^{s'}(\Omega)
\end{equation*}
is compact.
\end{theorem}

\subsection{Some results for the standard equation}\label{ss:se}
We recall some basic properties of the solutions to the following equation
\begin{equation}\label{e:standard}
(-\Delta)^{s}u+u-|u|^{p-1}u=0.
\end{equation}
The solutions of (\ref{e:standard}) are the critical points of $f_0$ given by (\ref{e:f0u}).
The non-degeneracy of the standard solution to Equation (\ref{e:standard}) is investigated by many works. For our purpose, we recall the following theorem. (For more results and details on this topic, see, for example, \cite{FS:URSFL}, \cite{FL:UNGS}, \cite{FQT:PSNFS}, \cite{FaVa13:UNPS} and the references therein.)
\begin{theorem}\label{t:unda}
There exists a unique solution (up to translation) $U\in H^{2s+1}(\mathbf R^n)$ to (\ref{e:standard}) such that
\begin{equation*}\label{e:uday}
\frac{C_1}{1+|x|^{n+2s}}\le U(x)\le \frac{C_2}{1+|x|^{n+2s}}, \quad \mbox{ for }\,x\in\mathbf R^n,
\end{equation*}
with some constants $0< C_1\le C_2$. Moreover, the linearized operator $L_0$ at $U$ is non-degenerate, that is, its kernel is given by
\begin{equation*}
{\rm ker} L_0={\rm span}\{\partial_{x_1}U,\cdots,\partial_{x_n}U\}.
\end{equation*}
\end{theorem}

\begin{remark}\label{r:pu}
By Lemma C.2 of \cite{FS:URSFL}, $\nabla U$ satisfies
\begin{equation*}\label{e:ps}
|\nabla U(x)|\le\frac{C}{1+|x|^{n+2s}},
\end{equation*}
for some constant $C$.
\end{remark}

\begin{remark}\label{r:lce}
The non-degeneracy of $L_0$ yields the coercivity estimate as follows:
\begin{equation*}
\langle L_0\phi,\phi\rangle_0 \ge C\|\phi\|_s^2 \quad \mbox{ for } \phi\perp K,
\end{equation*}
where $C$ is a positive constant, and $K$ is a suitable chosen $(n+1)$-dimensional subspace. For example, we can choose $K={\rm span}\{\phi_{-1},\partial_{x_1}U,\cdots,\partial_{x_n}U\}$ with $\phi_{-1}$ being the linear ground state of $L_0$. For more details, see \cite[Section 3]{FS:URSFL}.
\end{remark}

\subsection{Critical points of $F_{\varepsilon,\xi}$}\label{sb:cuf}
Let
\begin{equation}\label{d:a}
a=a(\xi)=(1+V(\xi))^{\frac{1}{2s}}
\end{equation}
and
\begin{equation}\label{d:b}
b=b(\xi)=[1+V(\xi)]^{\frac{1}{p-1}}.
\end{equation}
Then $bU(ax)$ solves (\ref{e:ss}). Set
\begin{equation}\label{e:dz}
z^{\varepsilon\xi}=b(\varepsilon\xi)U(a(\varepsilon\xi )x)
\end{equation}
and
\begin{equation*}
Z^{\varepsilon}=\left\{z^{\varepsilon\xi}(x-\xi)\,|\,\xi\in\mathbf R^n\right\}.
\end{equation*}
Therefore, every point in $ Z^{\varepsilon}$ is a critical point of (\ref{f:xi}) or, equivalently, a solution to Equation (\ref{e:ss}).
For simplicity, we will set $z=z_{\xi}=z_{\varepsilon,\xi}=z^{\varepsilon\xi}(x-\xi)$.

\section{Some estimates}\label{s:estimates}
In this section, we prove some useful estimates for future reference. From now on, $C$ denotes various constants.
\begin{lemma}\label{l:pxe}
Let $\bar\rho>0$. For $\varepsilon$ sufficiently small and $|\xi|\le \bar\rho$,
there holds
\begin{equation}\label{e:pxpxs}
\partial_{\xi_i}z^{\varepsilon\xi}=-\partial_{x_i}z^{\varepsilon\xi}(x-\xi)+O(\varepsilon), \quad\mbox{ in }H^s(\mathbf R^n).
\end{equation}
\end{lemma}
\begin{proof}
A direct calculation gives
\begin{eqnarray*}
&&\partial_{\xi_i}z^{\varepsilon\xi}(x-\xi)=\partial_{\xi_i}\left[b(\varepsilon\xi)U(a(\varepsilon\xi)(x-\xi))\right]\\
&=&\varepsilon[\partial_{\xi_i}b](\varepsilon\xi)U(a(\varepsilon\xi)(x-\xi))+\varepsilon b(\varepsilon\xi)[\partial_{\xi_i}a](\varepsilon\xi)[\nabla U](a(\varepsilon\xi)(x-\xi))\cdot(x-\xi)\\
&&-a(\varepsilon\xi)b(\varepsilon\xi)[\partial_{x_i}U](a(\varepsilon\xi)(x-\xi)):=Z_1+Z_2+Z_3.
\end{eqnarray*}
Note that
\begin{equation*}
Z_3=-a(\varepsilon\xi)b(\varepsilon\xi)[\partial_{x_i}U](b(\varepsilon\xi)(x-\xi))=-\partial_{x_i}z^{\varepsilon\xi}(x-\xi).
\end{equation*}
By the definition of $a, b$ and assumption of $V$, we have that
$$|a(\varepsilon\xi)|\le C,\quad |b(\varepsilon\xi)|<C,\quad \left|[\partial_{\xi_i} a](\varepsilon\xi)\right|\le C, \quad \left|[\partial_{\xi_i}b](\varepsilon\xi)\right|<C$$
for some constant $C$. From assumption of $V$,
\begin{equation*}\label{e:avx}
|a(\varepsilon\xi)|\ge 1,\quad |b(\varepsilon\xi)|\ge 1.
\end{equation*}
Therefore, from $\partial_{x_i}U(\cdot-\xi)\in H^s(\mathbf R^n)$, we have $Z_3\in H^s(\mathbf R^n)$.
By $U\in H^s(\mathbf R^n)$, it holds that
\begin{equation}\label{e:z1}
\|Z_1\|_s=O(\varepsilon)\|[\partial_{\xi_i}b](\varepsilon\xi)U(a(\varepsilon\xi)(\cdot-\xi))\|_s=O(\varepsilon).
\end{equation}
Since $\partial_{\xi_i}z^{\varepsilon\xi}\in H^s(\mathbf R^n)$ and $Z_1,\,Z_3\in H^s(\mathbf R^n)$, we have that $Z_2$ is also in $H^s(\mathbf R^n)$.
It follows that $[\nabla U](a(\varepsilon\xi)(\cdot-\xi))\cdot(\cdot-\xi)\in H^s(\mathbf R^n)$. So, we obtain that $[\nabla U](\cdot-\xi)\cdot(\cdot-\xi)\in H^s(\mathbf R^n)$. Again, by the property of $a$, it holds that
\begin{equation}\label{e:z2}
\|Z_2\|_s=O(\varepsilon).
\end{equation}
From (\ref{e:z1}) and (\ref{e:z2}), we have (\ref{e:pxpxs}). This completes the proof.
\end{proof}

\begin{lemma}\label{l:vxz}
Given $\bar\rho>0$ and small $\bar\varepsilon>0$, we have that, if $|\xi|\le \bar\rho$ and $0<\varepsilon<\bar\varepsilon$, then
\begin{equation*}
\int_{\mathbf R^n}|V(\varepsilon x)-V(\varepsilon\xi)|^2 z_{\xi}^2dx\le C (\varepsilon^2|\nabla V(\varepsilon\xi)|^2+\varepsilon^4),
\end{equation*}
and
\begin{equation*}
\int_{\mathbf R^n}|V(\varepsilon x)-V(\varepsilon\xi)|^2 |\partial_{x_i}z_{\xi}|^2dx\le C (\varepsilon^2|\nabla V(\varepsilon\xi)|^2+\varepsilon^4).
\end{equation*}
\end{lemma}

\begin{proof}
Since $V\in C^3_b(\mathbf R^n)$ implies that $|\nabla V(x)|\le C$ and $|D^2V(x)|\le C$, it holds that
\begin{equation*}
|V(\varepsilon x)-V(\varepsilon \xi)|\le \varepsilon |\nabla V(\varepsilon \xi)|\cdot |x-\xi|+C\varepsilon^2 |x-\xi|^2.
\end{equation*}
Therefore,
\begin{eqnarray*}
\int_{\mathbf R^n}|V(\varepsilon x)-V(\varepsilon \xi)|^2z_{\xi}^2dx&\le& C\varepsilon^2|\nabla V(\varepsilon \xi)|^2
\int_{\mathbf R^n}|x-\xi|^2z^2_{\xi}(x-\xi)dx\\
&&+C\varepsilon^4\int_{\mathbf R^n}|x-\xi|^4z^2_{\xi}(x-\xi)dx.
\end{eqnarray*}
By the definition of $z_{\xi}$,
\begin{eqnarray*}
\int_{\mathbf R^n}|x-\xi|^2z^2_{\xi}(x-\xi)dx&=&b^2(\varepsilon\xi)\int_{\mathbf R^n}|y|^2U^2(a(\varepsilon \xi)y)dy\\
&=&a^{-n-2}b^2\int_{\mathbf R^n}|y'|^2U^2(y')dy'.
\end{eqnarray*}
Using Theorem \ref{t:unda}, we obtain
\begin{equation*}
\int_{\mathbf R^n}|y'|^2U^2(y')dy'\le C_2\int_{\mathbf R^n}\frac{|y'|^2}{(1+|y'|)^{2n+4s}}\le C.
\end{equation*}
Since we assume $n> 4-4s$, it follows that
\begin{eqnarray*}
\int_{\mathbf R^n}|x-\xi|^4z^2_{\xi}(x-\xi)dx&\le& C_2\int_{\mathbf R^n}\frac{|x-\xi|^4}{(1+|x-\xi|)^{2n+4s}}dx\\
&\le& C_2\int_{\mathbf R^n}\frac{1}{(1+|x-\xi|)^{2n+4s-4}}dx\le C.
\end{eqnarray*}
Therefore, we get
\begin{equation}\label{e:vxvxi}
\int_{\mathbf R^n}|V(\varepsilon x)-V(\varepsilon \xi)|^2z_{\xi}^2dx\le C(\varepsilon^2|\nabla V(\varepsilon\xi)|^2+\varepsilon^4).
\end{equation}
For the second estimate, we have
\begin{eqnarray*}
\int_{\mathbf R^n}|V(\varepsilon x)-V(\varepsilon \xi)|^2|\partial_{x_i}z_{\xi}|^2dx&\le& C\varepsilon^2|\nabla V(\varepsilon \xi)|^2
\int_{\mathbf R^n}|x-\xi|^2|\partial_{x_i}z_{\xi}(x-\xi)|^2dx\\
&&+C\varepsilon^4\int_{\mathbf R^n}|x-\xi|^4|\partial_{x_i}z_{\xi}(x-\xi)|^2dx.
\end{eqnarray*}
By Remark \ref{r:pu}, $|\partial_{x_i}z_{\xi}(x-\xi)|\le \frac{C}{1+|x-\xi|^{n+2s}}$. Then a similar argument as the proof of (\ref{e:vxvxi}) gives
\begin{equation*}
\int_{\mathbf R^n}|V(\varepsilon x)-V(\varepsilon\xi)|^2 |\partial_{x_i}z_{\xi}|^2dx\le C (\varepsilon^2|\nabla V(\varepsilon\xi)|^2+\varepsilon^4).
\end{equation*}
This completes the proof.
\end{proof}

\begin{lemma}\label{l:dgss}
Given $\bar\rho>0$ and small $\bar\varepsilon>0$, it holds that, for $|\xi|\le \bar\rho$ and $0<\varepsilon<\bar\varepsilon$,
\begin{equation}\label{e:dg}
\|Df_{\varepsilon}(z_{\xi})\|_s\le C(\varepsilon|\nabla V(\varepsilon\xi)|+O(\varepsilon^2)),
\end{equation}
for some constant $C$.
\end{lemma}
\begin{proof}
Rewrite
\begin{equation*}
f_{\varepsilon}(u)=F^{\varepsilon \xi}(u)+\frac{1}{2}\int_{\mathbf R^n}(V(\varepsilon x)-V(\varepsilon \xi))u^2dx.
\end{equation*}
Since $z_{\xi}$ is a critical point of $F^{\varepsilon \xi}$, we get
\begin{eqnarray*}
\langle Df_{\varepsilon}(z_{\xi}),v\rangle_s &=& \langle DF^{\varepsilon \xi}(z_{\xi}),v\rangle_s+\int_{\mathbf R^n}(V(\varepsilon x)-V(\varepsilon \xi))z_{\xi}v dx\\
&=&\int_{\mathbf R^n}(V(\varepsilon x)-V(\varepsilon \xi))z_{\xi}v dx.
\end{eqnarray*}
By the H\"{o}lder inequality, we have
\begin{eqnarray*}
|\langle Df_{\varepsilon}(z_{\xi}),v\rangle|^2&\le& \|v\|_{0}^2\int_{\mathbf R^n}|V(\varepsilon x)-V(\varepsilon \xi)|^2z_{\xi}^2dx\\
&\le&\|v\|_s^2\int_{\mathbf R^n}|V(\varepsilon x)-V(\varepsilon \xi)|^2z_{\xi}^2dx.
\end{eqnarray*}
Then Lemma \ref{l:vxz} implies (\ref{e:dg}).
\end{proof}


\section{Invertibility}\label{s:inertibility}

In this section, we will discuss the invertibility of $D^2f_{\varepsilon}(z_{\xi})$ on $(T_{z_{\xi}}Z^{\varepsilon})^{\perp_s}$.
Here $T_{z_{\xi}}Z^{\varepsilon}$ is the tangent space to $Z^{\varepsilon}$ at $z_{\xi}$, and $(T_{z_{\xi}}Z^{\varepsilon})^{\perp_s}$ is the orthogonal complement of $T_{z_{\xi}}Z^{\varepsilon}$ in $H^s(\mathbf R^n)$.

Let
\begin{equation*}
\mathcal L_{\varepsilon,\xi}:(T_{z_{\xi}}Z^{\varepsilon})^{\perp_s}\to (T_{z_{\xi}}Z^{\varepsilon})^{\perp_s}
\end{equation*}
be the tangent operator of $Df_{\varepsilon}$ restricted on $(T_{z_{\xi}}Z^{\varepsilon})^{\perp_s}$, that is, on $(T_{z_{\xi}}Z^{\varepsilon})^{\perp_s}$,
\begin{equation*}
\langle \mathcal L_{\varepsilon,\xi}v,w\rangle_s=D^2f_{\varepsilon}(z_{\xi})[v,w].
\end{equation*}

The main aim of this section is to prove the following result which implies that $\mathcal L_{\varepsilon,\xi}$ is invertible on $(T_{z_{\xi}}Z^{\varepsilon})^{\perp_s}$.
\begin{prop}\label{p:invertible}
 Given $\bar\rho>0$, there exists $\bar\varepsilon>0$ such that, for all $|\xi|\le\bar\rho$ and $0<\varepsilon< \bar\varepsilon$, it holds that
\begin{equation*}
|\langle \mathcal L_{\varepsilon,\xi}v,v\rangle_s|\ge C\|v\|_s^2, \quad \forall v\in (T_{z_{\xi}}Z^{\varepsilon})^{\perp_s},
\end{equation*}
where $C>0$ is a constant only depending on $\bar\xi$ and $\bar\varepsilon$.
\end{prop}

Note that
\begin{equation*}
T_{z_{\xi}}Z^{\varepsilon}={\rm span}\{\partial_{\xi_1}z_{\xi},\cdots,\partial_{\xi_n}z_{\xi}\}.
\end{equation*}
By Lemma \ref{l:pxe}, we know that $\partial_{\xi_i}z_{\xi}$ is close to $-\partial_{x_i}z_{\xi}$ in $H^s(\mathbf R^n)$ when $\varepsilon\to 0$ and $|\xi|\le\bar\rho$. For convenience, we define
\begin{equation}\label{e:mxv}
K_{\varepsilon,\xi}={\rm span}\{z_{\xi},\partial_{x_1}z_{\xi},\cdots,\partial_{x_n}z_{\xi}\}.
\end{equation}

To prove Proposition \ref{p:invertible}, we need some lemmas.
\begin{lemma}
$z_{\xi}$ is a critical point of $F^{\varepsilon \xi}$ with Morse index one.
\end{lemma}
\begin{proof}
Since
\begin{equation}\label{e:D2F}
D^2F^{\varepsilon\xi}(z_{\xi})[z_{\xi},z_{\xi}]=-(p-1)\int_{\mathbf R^n}z_{\xi}^{p+1}dx<0,
\end{equation}
the operator $D^2F^{\varepsilon\xi}(z_{\xi})$ has at least one negative eigenvalue. For the details to prove that the Morse index of $z_{\xi}$ is one exactly,
see Section 3 in \cite{FS:URSFL}.
\end{proof}

\begin{lemma}\label{l:lzx}
Let $\bar\rho>0$, there exist $\varepsilon_0>0$ and a constant $C_1>0$ such that, for all $0<\varepsilon<\varepsilon_0$ and all $|\xi|\le\bar\rho$, it holds
\begin{equation*}
\langle \mathcal L_{\varepsilon,\xi}z_{\xi},z_{\xi}\rangle_s\le -C_1<0.
\end{equation*}
\end{lemma}
\begin{proof}
A direct calculus yields
\begin{equation}\label{e:lf}
\langle \mathcal L_{\varepsilon,\xi}z_{\xi},z_{\xi}\rangle_s=D^2F^{\varepsilon\xi}(z_{\xi})[z_{\xi},z_{\xi}]+\int_{\mathbf R^n}(V(\varepsilon x)-V(\varepsilon\xi))z_{\xi}^2dx.
\end{equation}
By (\ref{e:D2F}) and (\ref{e:dz}),
\begin{eqnarray*}
D^2F^{\varepsilon\xi}(z_{\xi})[z_{\xi},z_{\xi}]&=&-(p-1)\int_{\mathbf R^n}z_{\xi}^{p+1}dx\\
&=&-(p-1)\int_{\mathbf R^n}|b(\varepsilon\xi)U(a(\varepsilon\xi )(x-\xi))|^{p+1}dx\\
&=&-(p-1)[b(\varepsilon \xi)]^{p+1}[a(\varepsilon \xi)]^{-n}\int_{\mathbf R^n}U^{p+1}(x)dx
\end{eqnarray*}
From the definition of $a, b$ (see (\ref{d:a}) (\ref{d:b})) and $V(0)=0$, we have that,
for any fixed $\bar\rho>0$, there exists $\varepsilon_1>0$ small enough such that when $|\xi|\le\bar\rho$ and $0<\varepsilon<\varepsilon_1$, it holds
\begin{equation}\label{e:abr}
a(\varepsilon\xi)\in [\frac{1}{2},2] \quad \mbox{and}\quad b(\varepsilon\xi)\in [\frac{1}{2},2].
\end{equation}
Since $U$ is the unique solution (up to translation),
\begin{equation*}
\int_{\mathbf R^n}U^{p+1}(x)dx
\end{equation*}
is a constant. Therefore there is a positive constant $C_0$ such that
\begin{equation}\label{e:DFC}
D^2F^{\varepsilon\xi}(z_{\xi})[z_{\xi},z_{\xi}]\le -C_0<0.
\end{equation}
From Lemma \ref{l:vxz}, the second term on right side of (\ref{e:lf}) satisfies
\begin{eqnarray*}
&&\left|\int_{\mathbf R^n}(V(\varepsilon x)-V(\varepsilon\xi))z_{\xi}^2dx\right|\\
&\le&\int_{\mathbf R^n}(\varepsilon|\nabla V(\varepsilon\xi)\cdot (x-\xi)|+\varepsilon^2|D^2V(\eta)|\cdot |x-\xi|^2)z_{\xi}^2dx\\
&=&\varepsilon\int_{\mathbf R^n}|\nabla V(\varepsilon\xi)\cdot (x-\xi)|z_{\xi}^2dx+\varepsilon^2\int_{\mathbf R^n}|D^2V(\eta)|\cdot |x-\xi|^2z_{\xi}^2dx.
\end{eqnarray*}
Here $\eta$ is some point in $\mathbf R^n$.
Since $V\in C^3_b(\mathbf R^n)$, we have that
\begin{equation*}
|\nabla V(\varepsilon\xi)\cdot (x-\xi)|\le C|x-\xi|,
\end{equation*}
and
\begin{equation*}
|D^2V(\eta)|\cdot |x-\xi|^2\le C|x-\xi|^2.
\end{equation*}
Then by the definition of $z_{\xi}$,
\begin{eqnarray*}
&&\int_{\mathbf R^n}|\nabla V(\varepsilon\xi)\cdot (x-\xi)|z_{\xi}^2dx\\
&\le&C\int_{\mathbf R^n}|x-\xi||b(\varepsilon\xi)U(a(\varepsilon\xi )(x-\xi))|^2dx\\
&\le& C [b(\varepsilon\xi)]^2[a(\varepsilon\xi )]^{-n-1}\int_{\mathbf R^n}\frac{|x-\xi|}{(1+|x-\xi|)^{2n+4s}}dx
\end{eqnarray*}
Taking $|\xi|\le\rho$ and $\varepsilon<\varepsilon_1$ as in (\ref{e:abr}), we obtain that there exists a positive constant $C_2$ such that
\begin{equation*}
\int_{\mathbf R^n}|\nabla V(\varepsilon\xi)\cdot (x-\xi)|z_{\xi}^2dx<C_2
\end{equation*}
A similar argument yields
\begin{equation*}
\int_{\mathbf R^n}|D^2V(\eta)|\cdot |x-\xi|^2z_{\xi}^2dx
\le C_3\int_{\mathbf R^n}\frac{|x-\xi|^2}{1+|x-\xi|^{2n+4s}}dx\le C_4.
\end{equation*}
Therefore, when $|\xi|<\rho$ and $\varepsilon<\varepsilon_1$,
\begin{equation}\label{e:vv}
\left|\int_{\mathbf R^n}(V(\varepsilon x)-V(\varepsilon\xi))z_{\xi}^2dx\right|\le C_2\varepsilon+C_3\varepsilon^2.
\end{equation}
Then there is a $\varepsilon_0<\varepsilon_1$ such that when $\varepsilon<\varepsilon_0$,
\begin{equation}\label{e:cvv}
C_2\varepsilon+C_3\varepsilon^2<\frac{C_0}{2}
\end{equation}
From (\ref{e:DFC}), (\ref{e:vv}), (\ref{e:cvv}), we have
\begin{equation*}
\langle \mathcal L_{\varepsilon,\xi}z_{\xi},z_{\xi}\rangle_s\le -\frac{C_0}{2}<0.
\end{equation*}
This complete the proof.
\end{proof}

\begin{lemma}\label{l:d2f0i}
Let $\bar\rho>0$. There exists $\varepsilon_2>0$ small such that, for all $0<\varepsilon<\varepsilon_2$ and $|\xi|\le\bar\rho$, it holds
\begin{equation*}
D^2f_0(z_{\xi})[\phi,\phi]\ge C_2\|\phi\|_s^2,\quad \mbox{ for }\phi\in K_{\varepsilon,\xi}^{\perp_s},
\end{equation*}
where $C_2$ is a positive constant only depending on $\varepsilon_2$ and $\bar\rho$.
\end{lemma}

If this lemma does not hold, then there exists a sequence of $(\varepsilon_j,\xi_j)\to (0,\bar\xi)$ in $\mathbf R^+\times B_{\bar\rho}\subset\mathbf R^+\times \mathbf R^n$ and a sequence $\phi_j\in K_{\xi_j,\varepsilon_j}^{\perp_s}$
such that
\begin{equation}\label{e:pps1}
\|\phi_j\|_{s}=1,
\end{equation}
and
\begin{equation}\label{e:d2f0pj}
D^2f_0(z_{\xi_j})[\phi_j,\phi_j]\to 0,\quad\mbox{ as } j\to \infty.
\end{equation}
Since $\{\phi_j\}$ is bounded in $H^s(\mathbf R^n)$, we assume (passing to a subsequence) that $\phi_j$ converge weakly to a $\phi_{\infty}$ in $H^s(\mathbf R^n)$.
\begin{lemma}\label{l:pipk}
It holds that
\begin{equation*}
\phi_{\infty}\in K_{0,\bar\xi}^{\perp_s}.
\end{equation*}
\end{lemma}
\begin{proof}
Rewrite
\begin{eqnarray*}
\partial_{x_i}z_{\xi}&=&\partial_{x_i}U(x-\bar\xi)+\partial_{x_i}[b(\varepsilon_j\xi_j)U(a(\varepsilon_j\xi_j )(x-\xi_j))-U(x-\bar\xi)]\\
&:=&\partial_{x_i}U(x-\bar\xi)+\psi_j.
\end{eqnarray*}
By the definition of $a(\xi)$ and $b(\xi)$ (see (\ref{d:a}) and (\ref{d:b})), it holds that
\begin{eqnarray*}
\|\psi_j\|_s\to 0,\quad \mbox{as }j\to\infty.
\end{eqnarray*}
From $\phi_j \in K_{\xi_j,\varepsilon_j}^{\perp_s}$, it holds that
\begin{eqnarray*}
0=\langle \phi_j,\partial_{x_i}z_{\xi}\rangle_s
= \langle \phi_j,\partial_{x_i}U(\cdot-\bar\xi)\rangle_s+\langle\phi_j,\psi_j\rangle_s\to \langle \phi_{\infty},\partial_{x_i}U(\cdot-\bar\xi)\rangle_s.
\end{eqnarray*}
That is, $\phi_{\infty}\perp \partial_{x_i}U(\cdot-\bar\xi)$. Similarly, we have that $\phi_{\infty}\perp U(\cdot-\bar\xi)$. Therefore, we obtain $\phi_{\infty}\in K_{0,\bar\varepsilon}^{\perp_s}$. This completes the proof.
\end{proof}

Let
\begin{equation*}
\mathcal L_j:H^s(\mathbf R^n)\to H^s(\mathbf R^n)
\end{equation*}
be the operator given by
\begin{equation*}
\langle \mathcal L_j \phi,\psi\rangle_s=D^2f_0(z_{\xi_j})[\phi,\psi],\quad \mbox{ for }\phi,\psi\in H^s(\mathbf R^n),
\end{equation*}
and let
\begin{equation*}
\mathcal L_{\infty}:H^s(\mathbf R^n)\to H^s(\mathbf R^n)
\end{equation*}
be the operator defined by
\begin{equation*}
\langle \mathcal L_{\infty} \phi,\psi\rangle_s=D^2f_0(U(\cdot-\bar\xi))[\phi,\psi],\quad \mbox{ for }\phi,\psi\in H^s(\mathbf R^n).
\end{equation*}

We now have the following lemma.
\begin{lemma}
We have that $\phi_{\infty}=0$.
\end{lemma}
\begin{proof}
By (\ref{e:pps1}) and (\ref{e:d2f0pj}), we get that
\begin{equation*}
\langle \mathcal L_{j}\phi_j,\phi_j\rangle_s=\|\phi_j\|_s^2-p\int_{\mathbf R^n}z_{\xi_j}^{p-1}\phi_j^2dx\to 0.
\end{equation*}
and then
\begin{equation*}
p\int_{\mathbf R^n}z_{\xi_j}^{p-1}\phi_j^2dx\to 1.
\end{equation*}
Hence, from the definition of $z_{\xi}$ (see Section \ref{sb:cuf}), we obtain that
\begin{equation}\label{e:upi}
p\int_{\mathbf R^n}U^{p-1}(x-\bar\xi)\phi_j^2dx\to 1.
\end{equation}
Moreover, estimate
\begin{eqnarray*}
&&\left|\int_{\mathbf R^n}U^{p-1}(x-\bar\xi)(\phi_j^2-\phi_{\infty}^2)dx\right|\\
&\le& \left(\int_{\mathbf R^n}U^{2(p-1)}(x-\bar\xi)|\phi_j(x)-\phi_{\infty}(x)|^2dx\right)^{\frac{1}{2}}\|\phi_j+\phi_{\infty}\|_0\notag\\
&\le&C\left(\int_{\mathbf R^n}U^{2(p-1)}(x-\bar\xi)|\phi_j(x)-\phi_{\infty}(x)|^2dx\right)^{\frac{1}{2}}.\notag
\end{eqnarray*}
Let $B_r(\bar\xi)$ be the ball centered at $\bar\xi$ with radius $r$. Then
\begin{eqnarray}\label{e:bpjpi}
&&\int_{\mathbf R^n}U^{2(p-1)}(x-\bar\xi)|\phi_j(x)-\phi_{\infty}(x)|^2dx\notag\\
&=&\left(\int_{B_r(\bar\xi)}+\int_{\mathbf R^n\setminus B_r(\bar\xi)}\right)U^{2(p-1)}(x-\bar\xi)|\phi_j(x)-\phi_{\infty}(x)|^2dx.
\end{eqnarray}
For all sufficiently small $\epsilon>0$, there exists an $r(\epsilon)$ such that if $r>r(\epsilon)$, then, $U^{2(p-1)}(x-\bar\xi)<\epsilon$, for all $x\in \mathbf R^n\setminus B_r(\bar\xi)$. Thus,
\begin{equation*}
\left|\int_{\mathbf R^n\setminus B_r(\bar\xi)}U^{2(p-1)}(x-\bar\xi)|\phi_j(x)-\phi_{\infty}(x)|^2dx\right|\le \epsilon\|\phi_j(x)-\phi_{\infty}(x)\|_0^2.
\end{equation*}
We now estimate the other term in (\ref{e:bpjpi}). Let $\chi$ be a smooth function satisfying
\begin{equation*}
\chi(x)=\left\{\begin{array}{ll}
            1, & \mbox{for }x\in B_r(\bar\xi), \\
            0, & \mbox{for } x\in \mathbf R^n\setminus B_{r+1}(\bar\xi).
          \end{array}\right.
\end{equation*}
Then $\{\chi\phi_j\}$ is a bounded sequence in $H^{s}(B_{r+1}(\bar\xi))$. Therefore, there exists a function $\eta\in H^{s}(B_{r+1}(\bar\xi))$ such that, up to a subsequence, $\chi\phi_j\rightharpoonup \eta$. Since the embedding $H^{s}(B_{r+1}(\bar\xi))\hookrightarrow L^2(B_{r+1}(\bar\xi))$ is compact, we have $\chi\phi_j\to \eta$ in $L^2(B_{r+1}(\bar\xi))$. Then
$$\phi_j|_{B_r(\bar\xi)}=\chi\phi_j|_{B_r(\bar\xi)}\to \eta|_{B_r(\bar\xi)}, \quad \mbox{ in }L^2(B_r(\bar\xi)).$$
Since $\phi_j\rightharpoonup \phi_{\infty}$ in $L^2(B_r(\bar\xi))$, we obtain that
\begin{equation}\label{e:slpi}
\phi_j\to \phi_{\infty} \quad \mbox{in } L^2(B_r(\bar\xi)).
\end{equation}
It follows that
\begin{equation*}
\left|\int_{B_r(\bar\xi)}U^{2(p-1)}(x-\bar\xi)|\phi_j(x)-\phi_{\infty}(x)|^2dx\right|\to 0,\quad \mbox{as }j\to \infty.
\end{equation*}
By the arbitrary of $\epsilon$, we have that
\begin{equation*}
\int_{\mathbf R^n}U^{2(p-1)}(x-\bar\xi)|\phi_j(x)-\phi_{\infty}(x)|^2dx\to 0.
\end{equation*}
This yields that
\begin{equation}\label{e:upjpi0}
\int_{\mathbf R^n}U^{p-1}(x-\bar\xi)(\phi_j^2-\phi_{\infty}^2)dx\to 0.
\end{equation}
From (\ref{e:upi}) and (\ref{e:upjpi0}), we get that
\begin{equation*}
p\int_{\mathbf R^n}U^{p-1}(x-\bar\xi)\phi_{\infty}^2dx= 1.
\end{equation*}
On the other hand, by $\phi_j\rightharpoonup\phi_{\infty}$ in $H^s(\mathbf R^n)$, we have that
\begin{equation*}
\langle \phi_{\infty},\phi_{\infty}\rangle_s\leftarrow \langle \phi_j,\phi_{\infty}\rangle_s \le \|\phi_j\|_s\|\phi_{\infty}\|_s=\|\phi_{\infty}\|_s.
\end{equation*}
It follows that
\begin{equation*}
\|\phi_{\infty}\|_s\le 1.
\end{equation*}
Therefore, we obtain that
\begin{equation*}
\langle \mathcal L_{\infty}\phi_{\infty},\phi_{\infty}\rangle_s=\|\phi_{\infty}\|_s^2-p\int_{\mathbf R^n}U^{p-1}(x-\bar\xi)\phi_{\infty}^2dx\le  0.
\end{equation*}
By Theorem \ref{t:unda}, Remark \ref{r:lce} and Lemma \ref{l:pipk}, it holds that
\begin{equation*}
\langle \mathcal L_{\infty}\phi_{\infty},\phi_{\infty}\rangle_s\ge C\|\phi_{\infty}\|_s^2,
\end{equation*}
where $C$ is a positive constant.
Then we have that
\begin{equation*}
\|\phi_{\infty}\|_s=0.
\end{equation*}
This completes the proof.
\end{proof}

\begin{proof}[Proof of Lemma \ref{l:d2f0i}]
Note that $z_{\xi_j}^{p-1}$ decays uniformly to $0$ at infinity as $0<\varepsilon_j<\bar\varepsilon$ and $|\xi|\le \bar\rho$. Then, for any $\epsilon>0$, there exists a sufficiently large $r_0>0$ such that, for all $r>r_0$, $|z_{\xi_j}^{p-1}(x)|<\epsilon$ when $x\in \mathbf R^n\setminus B_r$. Therefore, from (\ref{e:slpi}) and $\phi_{\infty}=0$, we have that
\begin{eqnarray*}
\left|\int_{\mathbf R^n}z_{\xi_j}^{p-1}|\phi_j|^2dx\right|
\le C\int_{B_r}|\phi_j|^2dx+\epsilon\int_{\mathbf R^n\setminus B_r}|\phi_j|^2dx
\to\epsilon, \quad \mbox{as }j\to \infty.
\end{eqnarray*}
By the arbitrary of $\epsilon$, we have that
\begin{equation*}
\left|\int_{\mathbf R^n}z_{\xi_j}^{p-1}|\phi_j|^2dx\right|\to 0, \quad \mbox{as }j\to \infty.
\end{equation*}
Moreover, from (\ref{e:pps1}) and (\ref{e:d2f0pj}), it holds that
\begin{equation*}
0\leftarrow D^2f_0(z_{\xi_j})[\phi_j,\phi_j]=\|\phi_j\|_s^2-p\int_{\mathbf R^n}z_{\xi_j}^{p-1}|\phi_j|^2dx\to 1.
\end{equation*}
It is a contradiction. Thus we have Lemma \ref{l:d2f0i}.
\end{proof}

\begin{lemma}\label{l:df2kp}
Let $\bar\rho>0$. There exists $\varepsilon_3>0$ small such that for all $0<\varepsilon<\varepsilon_3$ and $|\xi|\le\bar\rho$, it holds
\begin{equation*}
D^2f_{\varepsilon}(z_{\xi})[\phi,\phi]\ge C_3\|\phi\|_s^2,\quad \mbox{ for }\phi\in K_{\varepsilon,\xi}^{\perp_s},
\end{equation*}
where $C_3$ is a positive constant only depending on $\varepsilon_2$ and $\bar\rho$.
\end{lemma}
\begin{proof}
By the nonnegativity of $V$ and Lemma \ref{l:d2f0i}, we have that, for all $\phi\in K_{\varepsilon,\xi}^{\perp}$
\begin{eqnarray*}
D^2f_{\varepsilon}(z_{\xi})[\phi,\phi]&=&D^2f_{0}(z_{\xi})[\phi,\phi]+\int_{\mathbf R^n}V(\varepsilon x)\phi^sdx\\
&\ge&D^2f_{0}(z_{\xi})[\phi,\phi]\ge C_0\|\phi\|_s^2.
\end{eqnarray*}
Here $0<\varepsilon<\varepsilon_2$ and $|\xi|\le\bar\rho$. Letting $\varepsilon_3=\varepsilon_2$ and $C_3=C_2$, we obtain the result.
\end{proof}

\begin{proof}[Proof of Proposition \ref{p:invertible}]
Let $\bar\varepsilon=\varepsilon_2$. From Lemma \ref{l:lzx}, Lemma \ref{l:df2kp}, Lemma \ref{l:pxe} and (\ref{e:mxv}), we have that, for all $|\xi|\le\bar\rho$ and $0<\varepsilon<\bar\varepsilon$,
\begin{equation*}
|\langle \mathcal L_{\varepsilon,\xi}v,v\rangle_s|\ge C\|v\|_s^2, \quad \forall v\in (T_{z_{\xi}}Z^{\varepsilon})^{\perp_s},
\end{equation*}
where $C>0$ is a constant only depending on $\bar\xi$ and $\bar\varepsilon$. This completes the proof.
\end{proof}


\section{Lyapunov-Schmidt Reduction}\label{s:reduction}

In this section, we will prove that the existence of critical points of $f_{\varepsilon}$ can be reduced to find critical points of an auxiliary finite dimensional functional.

\subsection{Auxiliary finite dimensional functional}

Let $P_{\varepsilon,\xi}$ be the orthogonal projection onto $(T_{z_{\xi}}Z^{\varepsilon})^{\perp_s}$. Our aim is to find a point $w\in (T_{z_{\xi}}Z^{\varepsilon})^{\perp_s}$ satisfying
\begin{equation}\label{e:pdg0}
P_{\varepsilon,\xi}Df_{\varepsilon}(z_{\xi}+w)=0.
\end{equation}
By expansion, we have that
\begin{equation*}
Df_{\varepsilon}(z_{\xi}+w)=Df_{\varepsilon}(z_{\xi})+D^2f_{\varepsilon}(z_{\xi})[w]+\mathcal R(z_{\xi},w).
\end{equation*}
Here the map $\mathcal R(z_{\xi},w)$ is given by
\begin{equation*}
\begin{array}{ccccl}
  \mathcal R(z_{\xi},w) & : & H^s & \to & \mathbf R \\
  \, & \, & v & \to & \int_{\mathbf R^n}R(z_{\xi},w)vdx,
\end{array}
\end{equation*}
where
\begin{equation*}
R(z_{\xi},w)=-(|z_{\xi}+w|^{p-1}(z_{\xi}+w)-|z_{\xi}|^{p-1}z_{\xi}-p|z_{\xi}|^{p-1}w).
\end{equation*}
\begin{lemma}\label{l:rw1w2}
For all $w_1,w_2\in B_1\subset H^{s}(\mathbf R^n)$, it holds that
\begin{equation*}
\|\mathcal R(z_{\xi},w_2)-\mathcal R(z_{\xi},w_1)\|_s\le C\max\{\|w_1\|_{s}^{\sigma},\|w_2\|_{s}^{\sigma}\}\|w_2-w_1\|_{s}.
\end{equation*}
where $\sigma=\min\{1,p-1\}$, $C$ is a constant independent on $w_1, w_2$. Here $B_1$ is the unit ball in $H^{s}(\mathbf R^n)$.
\end{lemma}
\begin{proof}
For all $v\in H^s(\mathbf R^n)$,
\begin{eqnarray*}
&&\left|[\mathcal R(z_{\xi},w_2)-\mathcal R(z_{\xi},w_1)](v)\right|\\
&\le&\int_{\mathbf R^n}\left||z_{\xi}+w_2|^{p-1}(z_{\xi}+w_2)-|z_{\xi}+w_1|^{p-1}(z_{\xi}+w_1)-p|z_{\xi}|^{p-1}(w_2-w_1)\right|\,|v|dx\notag\\
&\le&p\int_{\mathbf R^n}\left||z_{\xi}+w_1+\theta_1(w_2-w_1)|^{p-1}-|z_{\xi}|^{p-1})\right||w_2-w_1|\,|v|dx.\notag
\end{eqnarray*}
Here $\theta_1\in[0,1]$.
For $1<p\le 2$,
\begin{eqnarray*}
\left|[\mathcal R(z_{\xi},w_2)-\mathcal R(z_{\xi},w_1)](v)\right|&\le& p\int_{\mathbf R^n}|w_1+\theta_1(w_2-w_1)|^{p-1}|w_2-w_1|\,|v|dx\\
&\le& C\int_{\mathbf R^n}(|w_1|+|w_2|)^{p-1}|w_2-w_1|\,|v|dx\\
&\le& C (\|w_1\|_{L^{p+1}}^{p-1}+\|w_2\|_{L^{p+1}}^{p-1})\|w_2-w_1\|_{L^{p+1}}\|v\|_{L^{p+1}}.
\end{eqnarray*}
By Sobolev imbedding (Theorem \ref{l:embedding}), we have that
\begin{equation*}
H^{s}(\mathbf R^n)\hookrightarrow L^{p+1}(\mathbf R^n).
\end{equation*}
Therefore, we obtain that
\begin{equation*}
\left|[\mathcal R(z_{\xi},w_2)-\mathcal R(z_{\xi},w_1)](v)\right|\le C (\|w_1\|_{s}^{p-1}+\|w_2\|_{s}^{p-1})\|w_2-w_1\|_{s}\|v\|_{s}.
\end{equation*}
For $2<p<\frac{n+2s}{n-2s}$ (if $2<\frac{n+2s}{n-2s}$), it holds that
\begin{eqnarray*}
&&\left|[\mathcal R(z_{\xi},w_2)-\mathcal R(z_{\xi},w_1)](v)\right|\\
&=&C\int_{\mathbf R^n}|z_{\xi}+\theta_2(w_1+\theta_1(w_2-w_1))|^{p-2}|w_2-w_1|^2|v|dx\\
&\le&C\|z_{\xi}+\theta_2(w_1+\theta_1(w_2-w_1))\|_{L^{p+1}}^{p-2}\|w_2-w_1\|_{L^{p+1}}^2\|v\|_{L^{p+1}},
\end{eqnarray*}
where $\theta_2\in[0,1]$.
Similarly, by Sobolev imbedding, we have that
\begin{equation*}
\left|[\mathcal R(z_{\xi},w_2)-\mathcal R(z_{\xi},w_1](v)\right|\le C(\|z_{\xi}\|_{s}+\|w_1\|_{s}+\|w_2\|_{s})^{p-2} \|w_2-w_1\|_{s}^2\|v\|_s.
\end{equation*}
Therefore, we have
\begin{equation*}
\|\mathcal R(z_{\xi},w_2)-\mathcal R(z_{\xi},w_1)\|_{s}\le C\max(\|w_1\|_{s}^{\sigma},\|w_1\|_{s}^{\sigma})\|w_2-w_1\|_{s},
\end{equation*}
where $\sigma=\min\{1,p-1\}$. This completes the proof.
\end{proof}

\begin{cor}\label{l:Row}
It holds that $\|\mathcal R(z_{\xi},w)\|_s= O(\|w\|_{s}^{1+\sigma})$ where $\sigma=\min\{1,p-1\}$.
\end{cor}

\begin{proof}
Choosing $w_1=0$ and $w_2=w$ in Lemma \ref{l:rw1w2}, we find that
\begin{eqnarray*}
\|\mathcal R(z_{\xi},w)\|_s&\le& C(\|w\|_{s}^{1+\sigma}).
\end{eqnarray*}
\end{proof}

From the definition of $\mathcal L_{\varepsilon,\xi}$, Equation (\ref{e:pdg0}) becomes
\begin{equation}\label{e:lpr}
\mathcal L_{\varepsilon,\xi}w+P_{\varepsilon,\xi}Df_{\varepsilon}(z_{\xi})+P_{\varepsilon,\xi}\mathcal R(z_{\xi},w)=0,\quad \mbox{for }w\in (T_{z_{\xi}}Z)^{\perp_s}.
\end{equation}
By Proposition \ref{p:invertible}, we known that $\mathcal L_{\varepsilon,\xi}$ is invertible on $(T_{z_{\xi}}Z)^{\perp_s}$. Denote the invertible operator by $\mathcal L_{\varepsilon,\xi}^{-1}$. Then Equation (\ref{e:lpr}) is equivalent to
\begin{equation*}
w=N_{\varepsilon,\xi}(w).
\end{equation*}
Here
\begin{equation*}
N_{\varepsilon,\xi}(w)=-\mathcal L_{\varepsilon,\xi}^{-1}(P_{\varepsilon\xi}Df_{\varepsilon}(z_{\xi})+P_{\varepsilon\xi}R(z_{\xi},w)).
\end{equation*}
\begin{lemma}\label{l:nbd}
There is a small ball $B_{\delta}\subset (T_{z_{\xi}}Z^{\varepsilon})^{\perp_s}$ such that $N_{\varepsilon,\xi}$ maps $B_{\delta}$ to itself if $0<\varepsilon<\bar\varepsilon$ and $|\xi|\le \bar\rho$.
\end{lemma}
\begin{proof}
Using Lemma \ref{l:dgss}, we obtain
\begin{equation}\label{e:nexw}
\|N_{\varepsilon,\xi}(w)\|_{s}\le C(\varepsilon|\nabla V(\varepsilon\xi)|+O(\varepsilon^2))+O(\|w\|_{s}^{1+\sigma}).
\end{equation}
Then there is a small positive constant $\delta$ such that $N_{\varepsilon,\xi}$ maps $B_{\delta}\subset (T_{z_{\xi}}Z)^{\perp_s}$ to itself if $0<\varepsilon<\bar\varepsilon$ and $|\xi|\le \bar\rho$.
\end{proof}

\begin{lemma}\label{l:ncm}
For all $w_1,w_2\in B_1\subset H^{s}(\mathbf R^n)$, we have that
\begin{equation*}
\|N_{\varepsilon,\xi}(w_2)-N_{\varepsilon,\xi}(w_1)\|_{s}\le C \max(\|w_1\|_{s}^{\sigma},\|w_2\|_{s}^{\sigma})\|w_2-w_1\|_{s},
\end{equation*}
where $C$ is a constant independent on $w_1$ and $w_2$, $\sigma=\min\{1,p-1\}$.
\end{lemma}
\begin{proof}
Compute
\begin{eqnarray*}
\|N_{\varepsilon,\xi}(w_2)-N_{\varepsilon,\xi}(w_1)\|_{s}&=&\|-\mathcal L_{\varepsilon,\xi}^{-1}P_{\varepsilon\xi}(\mathcal R(z_{\xi},w_2)-\mathcal R(z_{\xi},w_1))\|_{s}\\
&\le&C\|\mathcal R(z_{\xi},w_2)-\mathcal R(z_{\xi},w_1)\|_{s}.
\end{eqnarray*}
Then by Lemma \ref{l:rw1w2}, we have that
\begin{equation*}
\|N_{\varepsilon,\xi}(w_2)-N_{\varepsilon,\xi}(w_1)\|_{s}\le C\max(\|w_1\|_{s}^{\sigma},\|w_1\|_{s}^{\sigma})\|w_2-w_1\|_{s},
\end{equation*}
where $\sigma=\min\{1,p-1\}$. This completes the proof.
\end{proof}

\begin{prop}\label{p:dgw}
For $0<\varepsilon<\bar\varepsilon$ and $|\xi|\le\bar\rho$, there exists a unique $w=w(\varepsilon,\xi)\in (T_{z_{\xi}}Z)^{\perp_s}$ such that
$Df_{\varepsilon}(z_{\xi}+w)\in T_{z_{\xi}}Z$, and $w(\varepsilon,\xi)$ is  of class $C^1$. Moreover, the functional
$\Phi_{\varepsilon}(\xi)=f_{\varepsilon}(z_{\xi}+w(\varepsilon,\xi))$ has the same regularity as $w$ and satisfies:
\begin{equation*}
\nabla\Phi_{\varepsilon}(\xi_0)=0\quad\Rightarrow\quad Df_{\varepsilon}(z_{\xi_0}+w(\varepsilon,\xi_0))=0.
\end{equation*}
\end{prop}
\begin{proof}
From Lemma \ref{l:nbd} and \ref{l:ncm}, the map $N_{\varepsilon,\xi}$ is a contraction on $B_{\delta}$ for $0<\varepsilon<\bar\varepsilon$ and $|\xi|\le\bar\rho$. Then there exists a unique $w$ such that $w=N_{\varepsilon,\xi}(w)$. For fixed $\varepsilon$, define
\begin{equation*}
\Xi_{\varepsilon}:(\xi,w)\to P_{\varepsilon,\xi}Df_{\varepsilon}(z_{\xi}+w).
\end{equation*}
Applying the Implicit Function Theorem to $\Xi_{\varepsilon}$, we have that $w(\varepsilon,\xi)$ is $C^1$ with respect to $\xi$. Then using a standard argument in \cite{ABC:ARMA97, AB:VPM}, we obtain that the critical points of $\Phi_{\varepsilon}(\xi)=f_{\varepsilon}(z_{\xi}+w(\varepsilon,\xi))$ give rise to critical points of $f_{\varepsilon}$.
\end{proof}

In what follows, we use the simple notation $w$ to denote $w(\varepsilon,\xi)$ which is obtained in Proposition \ref{p:dgw}.
\begin{remark}\label{r:ws}
By Equation (\ref{e:nexw}), it follows that
\begin{equation*}
\|w\|_{s}\le C(\varepsilon|\nabla V(\varepsilon\xi)|+\varepsilon^2),
\end{equation*}
where $C>0$.
\end{remark}

\begin{lemma}\label{l:nws}
The following inequality holds:
\begin{equation*}
\|\nabla_{\xi}w\|_{s}\le C\left(\varepsilon|\nabla V(\varepsilon\xi)|+O(\varepsilon^2)\right)^{\sigma},
\end{equation*}
where $C>0$ and $\sigma=\min\{1,p-1\}$.
\end{lemma}
\begin{proof}
By (\ref{e:lpr}) and Proposition \ref{p:dgw}, we have that, for all $v\in (T_{z_{\xi}}Z^{\varepsilon})^{\perp_s}$,
\begin{equation}\label{e:Lwe}
\langle \mathcal L_{\varepsilon,\xi}w,v\rangle_s+\langle Df_{\varepsilon}(z_{\xi}), v\rangle_s+\langle \mathcal R(z_{\xi},w),v\rangle_s=0.
\end{equation}
Since $DF_{\varepsilon,\xi}(z_{\xi})=0$, Equation (\ref{e:Lwe}) becomes
\begin{multline*}
\langle w,v\rangle_s+\int_{\mathbf R^n}V(\varepsilon x)w v dx-p\int_{\mathbf R^n}z_{\xi}^{p-1}wvdx+\int_{\mathbf R^n}[V(\varepsilon x)-V(\varepsilon \xi)]z_{\xi}vdx\\
+\int_{\mathbf R^n}R(z_{\xi},w)vdx=0.
\end{multline*}
Hence
\begin{eqnarray}\label{e:wjlr}
&&\langle \partial_{\xi_j}w,v\rangle_s+\int_{\mathbf R^n}V(\varepsilon x) (\partial_{\xi_j}w) v dx-p\int_{\mathbf R^n}z_{\xi}^{p-1}(\partial_{\xi_j}w)vdx\\
&&-p(p-1)\int_{\mathbf R^n}z_{\xi}^{p-2}(\partial_{\xi_j}z) wvdx+\int_{\mathbf R^n}(V(\varepsilon x)-V(\varepsilon \xi))(\partial_{\xi_j}z) vdx
\notag\\
&&-\varepsilon(\partial_{x_j}V)(\varepsilon \xi)\int_{\mathbf R^n}zvdx
-\int_{\mathbf R^n}(R_z\partial_{\xi_j}z+R_w\partial_{\xi_j}w)vdx=0.\notag
\end{eqnarray}
Set $\hat {\mathcal L}=\mathcal L_{\varepsilon,\xi}-\mathcal{R}_w$, where $\langle \mathcal{R}_w v_1,v_2\rangle=\int_{\mathbf R^n}R_{w}v_1v_2dx.$ Since $R_w\to 0$ as $w\to 0$ and $\mathcal L_{\varepsilon,\xi}$ is invertible on $(T_{z_{\xi}}Z^{\varepsilon})^{\perp_s}$, $\hat {\mathcal L}$ is also invertible for $0<\varepsilon<\bar\varepsilon$ and $|\xi|\le\bar\rho$. From (\ref{e:wjlr}), it holds that
\begin{multline*}
\langle\hat {\mathcal L} \partial_{\xi_j}w,v\rangle=p(p-1)\int_{\mathbf R^n}z_{\xi}^{p-2}(\partial_{\xi_j}z) wvdx-\int_{\mathbf R^n}(V(\varepsilon x)-V(\varepsilon \xi))(\partial_{\xi_j}z) vdx\\
+\varepsilon(\partial_{x_j}V)(\varepsilon \xi)\int_{\mathbf R^n}zvdx
+\int_{\mathbf R^n}R_z\partial_{\xi_j}zvdx=T_1+T_2+T_3+T_4.
\end{multline*}
Next, we shall estimate every term on the left of the equation above.
By Theorem \ref{t:unda} and Remark \ref{r:pu}, it holds that, for $1<p\le 2$,
\begin{eqnarray*}
|T_1|&=&p(p-1)\left|\int_{\mathbf R^n}z_{\xi}^{p-2}(\partial_{\xi_j}z) wvdx\right|\\
&\le &C\int_{\mathbf R^n}(1+|x|^{n+2s})^{2-p}\,\frac{1}{1+|x|^{n+2s}}|wv|dx\\
&\le&C\int_{\mathbf R^n}\frac{1}{(1+|x|^{n+2s})^{p-1}}|wv|dx\\
&\le& C\int_{\mathbf R^n}|wv|dx\le C\|w\|_0\|v\|_0\le C\|w\|_{s}\|v\|_s,
\end{eqnarray*}
and, for $2<p<\frac{n+2s}{n-2s}$ (if $2<\frac{n+2s}{n-2s}$),
\begin{eqnarray*}
\left|\int_{\mathbf R^n}z_{\xi}^{p-2}(\partial_{\xi_j}z) wvdx\right|&\le &C\int_{\mathbf R^n}\frac{1}{(1+|x|^{n+2s})^{p-1}}|wv|dx\\
&\le&C\|w\|_{s}\|v\|_s.
\end{eqnarray*}
Therefore, we have that
\begin{equation*}
|T_1|\le C\|w\|_{s}\|v\|_s.
\end{equation*}
Since $0<\varepsilon<\bar\varepsilon$ and $|\xi|\le \bar \rho$, by Lemma \ref{l:vxz} we have
\begin{eqnarray}\label{e:t2}
|T_2|&=&\left|\int_{\mathbf R^n}(V(\varepsilon x)-V(\varepsilon \xi))(\partial_{\xi_j}z) vdx\right|\\
&\le&\int_{\mathbf R^n}|V(\varepsilon x)-V(\varepsilon \xi)||\partial_{\xi_j}z| |v|dx\notag\\
&\le&\left(\int_{\mathbf R^n}|V(\varepsilon x)-V(\varepsilon \xi)|^2|\partial_{\xi_j}z|^2dx\right)^{\frac{1}{2}}\|v\|_0\notag\\
&\le& C \varepsilon|\nabla V(\varepsilon\xi)|\|v\|_s.\notag
\end{eqnarray}
Then we obtain that
\begin{equation*}
|T_2|\le C\varepsilon|\nabla V(\varepsilon\xi)| \|v\|_s.
\end{equation*}
Estimating the third term, we have
\begin{multline*}
|T_3|=\varepsilon\left|(\partial_{x_j}V)(\varepsilon \xi)\int_{\mathbf R^n}zvdx\right|
\le\varepsilon|(\nabla V)(\varepsilon \xi)|\|z\|_0\|v\|_0
\le \varepsilon|(\nabla V)(\varepsilon \xi)|\|v\|_s.
\end{multline*}
It remains to estimate the final term.
A direct computation yields
\begin{eqnarray*}
 |T_4|&=&\left|\int_{\mathbf R^n}R_z\partial_{\xi_j}zvdx\right|\le \int_{\mathbf R^n}|R_z| |\partial_{\xi_j}z||v|dx\\
 &\le&C\int_{\mathbf R^n}\left||z_{\xi}+w|^{p-1}-|z_{\xi}|^{p-1}\right|\cdot|\partial_{\xi_j}z_{\xi}|\cdot|v|dx\\
 &&+C\int_{\mathbf R^n}|z_{\xi}|^{p-2}\cdot|\partial_{\xi_j}z_{\xi}|\cdot|w|\cdot|v|dx
\end{eqnarray*}
Then, for $1<p\le 2$,
\begin{eqnarray*}
&&\int_{\mathbf R^n}\left||z_{\xi}+w|^{p-1}-|z_{\xi}|^{p-1}\right|\cdot|\partial_{\xi_j}z_{\xi}|\cdot|v|dx\\
 &\le&C\int_{\mathbf R^n}|w|^{p-1}\cdot|\partial_{\xi_j}z_{\xi}|\cdot|v|dx\\
 &\le& C\|w\|_{L^{p+1}}^{p-1}\|\partial_{\xi_j}z_{\xi}\|_{L^{p+1}}\|v\|_{L^{p+1}}\le C\|w\|_{s}^{p-1}\|v\|_{s},
\end{eqnarray*}
and, for $2<p<\frac{n+2s}{n-2s}$ (if $2<\frac{n+2s}{n-2s}$),
\begin{eqnarray*}
&&\int_{\mathbf R^n}\left||z_{\xi}+w|^{p-1}-|z_{\xi}|^{p-1}\right|\cdot|\partial_{\xi_j}z_{\xi}|\cdot|v|dx\\
&\le& \int_{\mathbf R^n}(p-1)|z_{\xi}+\theta_3w|^{p-2}|w|\cdot|\partial_{\xi_j}z_{\xi}|\cdot|v|dx\\
&\le& C\|z_{\xi}+\theta_3w\|_{L^{p+1}}^{p-2}\|\partial_{\xi_j}z_{\xi}\|_{L^{p+1}}\|w\|_{L^{p+1}}\|v\|_{L^{p+1}}\le C\|w\|_{s}\|v\|_{s}.
\end{eqnarray*}
Here $\theta_3\in[0,1]$.
Then we have that
\begin{equation*}
\int_{\mathbf R^n}\left||z_{\xi}+w|^{p-1}-|z_{\xi}|^{p-1}\right|\cdot|\partial_{\xi_j}z_{\xi}|\cdot|v|dx\le C\|w\|_{s}^{\sigma}\|v\|_{s},
\end{equation*}
where $\sigma=\min\{1,p-1\}$.
Furthermore, we estimate
\begin{eqnarray*}
&&\int_{\mathbf R^n}|z_{\xi}|^{p-2}\cdot|\partial_{\xi_j}z_{\xi}|\cdot|w|\cdot|v|dx\\
&\le&C\int_{\mathbf R^n}\left(\frac{1}{(1+|x-\xi|)^{n+2s}}\right)^{p-1}\cdot|w|\cdot|v|dx\\
&\le& C\|w\|_0\|v\|_0\le C\|w\|_{s}\|v\|_s.
\end{eqnarray*}
Therefore, we obtain
\begin{equation*}
|T_4|\le C\|w\|_{s}^{\sigma}\|v\|_{s},
\end{equation*}
where $\sigma=\min\{1,p-1\}$.

Summarizing the estimates for $T_1,T_2,T_3,T_4$, we get
\begin{equation*}
\|\hat{\mathcal L} \partial_{\xi_j}w\|_s\le C(\varepsilon|\nabla V(\varepsilon\xi)|+\|w\|_{s}^{\sigma}).
\end{equation*}
Then by Remark \ref{r:ws}, it holds that
\begin{equation*}
\|\hat{\mathcal L} \partial_{\xi_j}w\|_s\le C(\varepsilon|\nabla V(\varepsilon\xi)|+O(\varepsilon^2))^{\sigma}.
\end{equation*}
 Thus, we finally obtain
\begin{equation*}
\|\nabla_{\xi}w\|_{s}\le C(\varepsilon|\nabla V(\varepsilon\xi)|+O(\varepsilon^2))^{\sigma}.
\end{equation*}
This completes the proof.
\end{proof}

\subsection{Analysis of $\Phi_{\varepsilon}(\xi)$}

In this subsection, we shall expand $\Phi_{\varepsilon}(\xi)$. By the definition, we have that
\begin{eqnarray*}
\Phi_{\varepsilon}(\xi)&=&\frac{1}{2}\|z_{\xi}+w(\varepsilon,\xi)\|_s^2+\frac{1}{2}\int_{\mathbf R^n}V(\varepsilon x)(z_{\xi}+w(\varepsilon,\xi))^2dx\\
&&-\frac{1}{p+1}\int_{\mathbf R^n}|z_{\xi}+w(\varepsilon,\xi)|^{p+1}dx
\end{eqnarray*}
Since $(-\Delta)^sz_{\xi}+z_{\xi}+V(\varepsilon\xi)z_{\xi}=z_{\xi}^p$, it holds that
\begin{equation*}
\langle z_{\xi},w\rangle_{s}=-V(\varepsilon\xi)\int_{\mathbf R^n}z_{\xi}wdx+\int_{\mathbf R^n}z_{\xi}^pwdx.
\end{equation*}
Therefore, we can rewrite
\begin{eqnarray*}
\Phi_{\varepsilon}(\xi)&=&\left(\frac{1}{2}-\frac{1}{p+1}\right)\int_{\mathbf R^n}z^{p+1}dx+\frac{1}{2}\int_{\mathbf R^n}(V(\varepsilon x)-V(\varepsilon\xi))z^2dx\\
&&+\int_{\mathbf R^n}(V(\varepsilon x)-V(\varepsilon\xi))zwdx+\frac{1}{2}\int_{\mathbf R^n}V(\varepsilon x)w^2dx\\
&&+\frac{1}{2}\|w\|_s^2-\frac{1}{p+1}\int_{\mathbf R^n}\left(|z+w|^{p+1}-z^{p+1}-(p+1)z^pw\right)dx.
\end{eqnarray*}
By the definition of $z(x)$ (see Subsection \ref{sb:cuf}), $z(x)=b(\varepsilon\xi)U(a(\varepsilon\xi )x)$ where $a(\varepsilon \xi)=(1+V(\varepsilon\xi))^{\frac{1}{2s}}$ and $b(\varepsilon \xi)=(1+V(\varepsilon\xi))^{\frac{1}{p-1}}$. Then we have that
\begin{equation*}
\int_{\mathbf R^n}z^{p+1}dx=C_0(1+V(\varepsilon\xi))^{\theta},
\end{equation*}
where $C_0=\int_{\mathbf R^n}U^{p+1}dx$ and $\theta=\frac{p+1}{p-1}-\frac{n}{2s}$. Let $C_1=\left(\frac{1}{2}-\frac{1}{p+1}\right)C_0$. Then
\begin{equation*}
\Phi_{\varepsilon}(\xi)=C_1(1+V(\varepsilon\xi))^{\theta}+\Gamma_{\varepsilon}(\xi)+\Psi_{\varepsilon}(\xi),
\end{equation*}
where
\begin{equation*}
\Gamma_{\varepsilon}(\xi)=\frac{1}{2}\int_{\mathbf R^n}[V(\varepsilon x)-V(\varepsilon\xi)]z^2dx+\int_{\mathbf R^n}[V(\varepsilon x)-V(\varepsilon\xi)]zwdx
\end{equation*}
and
\begin{eqnarray*}
\Psi_{\varepsilon}(\xi)&=&\frac{1}{2}\int_{\mathbf R^n}V(\varepsilon x)w^2dx+\frac{1}{2}\|w\|_s^2\\
&&-\frac{1}{p+1}\int_{\mathbf R^n}\left[|z+w|^{p+1}-z^{p+1}-(p+1)z^pw\right]dx.\notag
\end{eqnarray*}

\begin{lemma}
We have the following estimate:
\begin{equation*}
|\nabla \Psi_{\varepsilon}(\xi)|\le C \|w\|_s( \|w\|_s^{\sigma}+\|\nabla_{\xi}w\|_s).
\end{equation*}
\end{lemma}
\begin{proof}
A direct calculus yields, for $j=1,2,\cdots,n$,
\begin{multline}\label{e:psif}
  \left|\partial_{\xi_j}\left( \frac{1}{2}\int_{\mathbf R^n}V(\varepsilon x)w^2dx+\frac{1}{2}\|w\|_s^2\right)\right|=\left|\int_{\mathbf R^n}V(\varepsilon x)w\partial_{\xi_j} wdx+\langle w,\partial_{\xi_j} w\rangle_s \right|\\
  \le C(\|w\|_{0}\|\partial_{\xi_j}w\|_{0}+\|w\|_s\|\partial_{\xi_j}w\|_s)\le C(\|w\|_s\|\partial_{\xi_j}w\|_s).
\end{multline}
Estimate
\begin{eqnarray*}
&&\left|\partial_{\xi_j}\left(\frac{1}{p+1}\int_{\mathbf R^n}\left(|z+w|^{p+1}-z^{p+1}-(p+1)z^pw\right)dx\right)\right|\\
&=&\left|\int_{\mathbf R^n}\left(|z+w|^{p}(\partial_{\xi_j}z+\partial_{\xi_j}w)-z^p(\partial_{\xi_j}z+\partial_{\xi_j}w)-pz^{p-1}w\partial_{\xi_j}z\right)dx\right|\\
&=&\left|\int_{\mathbf R^n}\left(p|z+\theta_4w|^{p-1}w(\partial_{\xi_j}z+\partial_{\xi_j}w)-pz^{p-1}w\partial_{\xi_j}z\right)dx\right|\\
&=&\left|\int_{\mathbf R^n}\left(pw(|z+\theta_4w|^{p-1}-z^{p-1})\partial_{\xi_j}z+pw|z+\theta_4w|^{p-1}\partial_{\xi_j}w\right)dx\right|.
\end{eqnarray*}
Here $\theta_4\in[0,1]$.
Then, for $1<p\le 2$,
\begin{eqnarray*}
&&\left|\int_{\mathbf R^n}\left(pw(|z+\theta_4w|^{p-1}-z^{p-1})\partial_{\xi_j}zdx\right)\right|\\
&\le&\int_{\mathbf R^n}\left|pw^p\partial_{\xi_j}z\right|dx\le C\|\partial_{\xi_j}z\|_{L^{p+1}}\|w\|_{L^{p+1}}^p\le C\|w\|_s^p,
\end{eqnarray*}
and, for $2<p<\frac{n+2s}{n-2s}$ (if $2<\frac{n+2s}{n-2s}$),
\begin{eqnarray*}
&&\left|\int_{\mathbf R^n}\left(pw|(z+\theta_4w|^{p-1}-z^{p-1})\partial_{\xi_j}zdx\right)\right|\\
&\le& \int_{\mathbf R^n}\left|p(p-1)w^2|z+\theta_5w|^{p-2}\partial_{\xi_j}z\right|dx\\
&\le& C\|\partial_{\xi_j}z\|_{L^{p+1}}\,\|z+\theta_5w\|_{L^{p+1}}^{p-2}\,\|w\|_{L^{p+1}}^2\\
&\le& C\|\partial_{\xi_j}z\|_s\,\|z+\theta_5w\|_{s}^{p-2}\,\|w\|_{s}^2\le C\|w\|_{s}^2.
\end{eqnarray*}
Here $\theta_5\in [0,1]$.
Therefore,
\begin{equation*}
\left|\partial_{\xi_j}\left(\frac{1}{p+1}\int_{\mathbf R^n}\left(|z+w|^{p+1}-z^{p+1}-(p+1)z^pw\right)dx\right)\right|\le C\|w\|_s^{1+\sigma}.
\end{equation*}
Moreover,
\begin{equation*}
\left|\int_{\mathbf R^n}pw|z+\theta_4w|^{p-1}\partial_{\xi_j}wdx\right|\le C\|z+\theta_4w\|_{L^{p+1}}^{p-1}\|w\|_{L^{p+1}}\|\partial_{\xi_j}\|_{L^{p+1}}
\le C\|w\|_{s}\|\partial_{\xi_j}w\|_{s}.
\end{equation*}
Therefore, we have that
\begin{equation*}
|\nabla \Psi_{\varepsilon}(\xi)|\le C \|w\|_s\left( \|w\|_s^{\sigma}+\|\nabla_{\xi}w\|_s\right).
\end{equation*}
This completes the proof.
\end{proof}
\begin{lemma}
It holds
\begin{equation}\label{e:gvz}
|\nabla \Gamma_{\varepsilon}(\xi)|\le C\varepsilon^{1+\sigma}.
\end{equation}
\end{lemma}
\begin{proof}
Compute
\begin{eqnarray*}
&&\int_{\mathbf R^n}(V(\varepsilon x)-V(\varepsilon\xi))z^2dx\\
&=&\varepsilon\int_{\mathbf R^n}\nabla V(\varepsilon\xi)\cdot(x-\xi)z^2dx\\
&&+\varepsilon^2\int_{\mathbf R^n}D^2V(\varepsilon\xi+\theta_6(\varepsilon-\xi))[x-\xi,x-\xi]z^2dx\\
&=&\varepsilon\int_{\mathbf R^n}\nabla V(\varepsilon\xi)\cdot yz^2(y)dx\\
&&+\varepsilon^2\int_{\mathbf R^n}D^2V(\varepsilon\xi+\theta_6(\varepsilon-\xi))[x-\xi,x-\xi]z^2dx\\
&=&\varepsilon^2\int_{\mathbf R^n}D^2V(\varepsilon\xi+\theta_6(\varepsilon-\xi))[x-\xi,x-\xi]z^2dx
\end{eqnarray*}
where $\theta_6\in[0,1]$.
Since $V\in C^3_b(\mathbf R^n)$, it holds that
\begin{eqnarray}\label{e:vzn1}
&&\left|\partial_{\xi_j}\left(\int_{\mathbf R^n}(V(\varepsilon x)-V(\varepsilon\xi))z^2dx\right)\right|\\
&=&\varepsilon^2\left|\partial_{\xi_j}\left(\int_{\mathbf R^n}D^2V(\varepsilon\xi+\theta_6(\varepsilon-\xi))[x-\xi,x-\xi]z^2dx\right)\right|
\le C\varepsilon^2.\notag
\end{eqnarray}
Estimate
\begin{eqnarray*}
&&\left|\partial_{\xi_j}\int_{\mathbf R^n}[V(\varepsilon x)-V(\varepsilon\xi)]zwdx\right|\\
&\le&\varepsilon|\nabla V(\varepsilon\xi)|\int_{\mathbf R^n}|zw|dx+\int_{\mathbf R^n}|V(\varepsilon x)-V(\varepsilon\xi)||\partial_{\xi_j}z||w|dx\\
&&+\int_{\mathbf R^n}|V(\varepsilon x)-V(\varepsilon\xi)||z||\partial_{\xi_j}w|dx\\
&\le&\varepsilon|\nabla V(\varepsilon\xi)|\|w\|_0+\left(\int_{\mathbf R^n}|V(\varepsilon x)-V(\varepsilon\xi)|^2|\partial_{\xi_j}z|^2dx\right)^{\frac{1}{2}}\|w\|_0\\
&&+\left(\int_{\mathbf R^n}|V(\varepsilon x)-V(\varepsilon\xi)|^2|z|^2dx\right)^{\frac{1}{2}}\|\partial_{\xi_j}w\|_0.
\end{eqnarray*}
Thus by Lemma \ref{l:vxz}, Remark \ref{r:ws} and Lemma \ref{l:nws}, we have that
\begin{equation}\label{e:vzn2}
\left|\nabla\left(\int_{\mathbf R^n}(V(\varepsilon x)-V(\varepsilon\xi))zwdx\right)\right|\le C\varepsilon(\varepsilon+\|w\|_s+\|\nabla w\|_s)\le C \varepsilon^{1+\sigma}.
\end{equation}
Therefore, from Estimates (\ref{e:vzn1}) (\ref{e:vzn2}), Equation (\ref{e:gvz}) holds.
\end{proof}

Let $\alpha(\varepsilon,\xi)=\theta C_1(1+V(\varepsilon\xi))^{\theta-1}$, where $\theta=\frac{p+1}{p-1}-\frac{n}{2s}$. Then summarizing all conclusion above, we get the following proposition.
\begin{prop}\label{p:prv}
It holds
\begin{equation*}
\nabla\Phi_{\varepsilon}(\xi)=\alpha(\varepsilon\xi)\varepsilon \nabla V(\varepsilon\xi)+\varepsilon^{1+\sigma}\varpi_{\varepsilon}(\xi),
\end{equation*}
where $\varpi_{\varepsilon}(\xi)$ is a bounded function and $\sigma=\min\{1,p-1\}$.
\end{prop}
\begin{remark}\label{r:prv}
Using similar argument, we can prove that
\begin{equation*}
\Phi_{\varepsilon}(\xi)=C(1+V(\varepsilon\xi))^{\theta}+\gamma_{\varepsilon}(\xi),
\end{equation*}
where $C>0$, $\theta=\frac{p+1}{p-1}-\frac{n}{2s}$ and $|\gamma_{\varepsilon}(\xi)|\le C(\varepsilon |\nabla V(\varepsilon\xi)|+\varepsilon^2)$.
\end{remark}


\section{Proof of the main theorem}\label{s:proof}
In this section, we shall prove the main theorem by a classical perturbation result.
\subsection{A multiplicity result by perturbation}\label{ss:abstract}
Let $M\subset\mathbf R^n$ be a non-empty set. We denote by $M_{\delta}$ its $\delta$-neighbourhood.
The cup length $l(M)$ of $M$ is defined by
\begin{equation*}
l(M)=1+\sup\{k\in \mathbf N\,\mid\, \exists \alpha_1,\cdots,\alpha_k\in \check{H}^*(M)\setminus 1, \alpha_1\cup\cdots\cup\alpha_k\ne 0\}.
\end{equation*}
If no such class exists, we set $l(M)=1$. Here $\check{H}^*(M)$ is the Alexander cohomology of $M$ with real coefficients and $\cup$ denotes the cup product.

Assume that $V$ has a smooth manifold of critical points of $M$. According to Bott \cite{Bo:AM57}, we say that $M$ is non-degenerate critical manifold for $V$ if every $x\in M$ is a critical point of $V$ and the nullity of all $x\in M$ equals to the dimension of $M$. 

Now we recall a classical perturbation result. For more details, see Theorem 6.4 of Chapter II in \cite{Ch:IDMT}.
\begin{theorem}\label{t:abstract}
Let $h\in C^2(\mathbf R^n)$ and $\Sigma\subset \mathbf R^n$ be a smooth compact non-degenerate critical manifold of $h$. Let $W$ be a neighbourhood of $\Sigma$ and let $g\in C^1(\mathbf R^n)$. Then, if $\|h-g\|_{C^1(\overline W)}$ is sufficiently small, the function $g$ has at least $l(\Sigma)$ critical points in $W$.
\end{theorem}

\subsection{Proof of Theorem \ref{t:main}}
With the preliminary considerations of the sections above, we now prove Theorem \ref{t:main} by the abstract perturbation theorem above.

\begin{proof}[Proof of Theorem \ref{t:main}]
Fix $\bar\rho>0$ such that $M\subset B_{\bar \rho}$. Since $M$ is a non-degenerate smooth critical manifold of $V$, it is a non-degenerate critical manifold of $C_1(1+V)^{\theta}$ as well. To use Theorem \ref{t:abstract}, we define
\begin{equation*}
h(\xi)=C_1(1+V(\xi))^{\theta},
\end{equation*}
and
\begin{equation*}
 g(\xi)=\Phi_{\varepsilon}\left(\frac{\xi}{\varepsilon}\right).
\end{equation*}
Set $\Sigma=M$. Fix a $\delta$-neighbourhood $M_{\delta}$ of $M$ such that $M_{\delta}\subset B_{\bar\rho}$ and the only critical points of $V$ in $M_{\delta}$ are those of $M$. Let $W=M_{\delta}$. From Proposition \ref{p:prv} and Remark \ref{r:prv}, the function $\Phi_{\varepsilon}(\cdot/\varepsilon)$ is converges to $h(\cdot)$ in $C^1(\overline W)$ as $\varepsilon\to 0$. Then Theorem \ref{t:abstract} yields the existence of at least $l(M)$ critical points of $g$ for $\varepsilon$ sufficiently small.

Let $\xi_k\in M_{\delta}$ be any of those critical points. Then $\xi_k/\varepsilon$ is a critical point of $\Phi_{\varepsilon}$ and Proposition \ref{p:dgw} implies that
\begin{equation*}
u_{\varepsilon,\xi_k}(x)=z_{\xi_k}\left(x-\frac{\xi_k}{\varepsilon}\right)+w(\varepsilon,\xi_k)
\end{equation*}
is a critical point of $f_{\varepsilon}$ and hence a solution of Equation (\ref{e:change}). Thus
\begin{equation*}
u_{\varepsilon,\xi_k}\left(\frac{x}{\varepsilon}\right)\simeq z_{\xi_k}\left(\frac{x-\xi_k}{\varepsilon}\right)
\end{equation*}
is a solution of Equation (\ref{e:main-equation}).

 Any $\xi_k$ converges to some point $\xi_k^*\in M_{\delta}$ as $\varepsilon\to 0$. By Proposition \ref{p:prv}, we have that $\xi_k^*$ is a stationary point of $V$. Then the choice of $M_{\delta}$ implies that $\xi_k^*\in M$. That is, $u_{\varepsilon,\xi_k}(x/\varepsilon)$ concentrates near a point of $M$. This completes the proof.
\end{proof}

\section*{Acknowledgments}
We are grateful to the anonymous referees for useful comments and suggestions. This work was supported by National Natural Science Foundation of China (No. 11401521) and Zhejiang Provincial Natural Science Foundation of China (LQ13A010003).


\newcommand{\Toappear}{to appear in}

\bibliography{mrabbrev,cz_abbr2003-0,localbib}

\bibliographystyle{plain}
\end{document}